\documentclass[preprint]{elsarticle}
\usepackage[top=2.5cm,bottom=2.5cm,left=2.5cm,right=2.5cm]{geometry}

\usepackage{amssymb}
\usepackage{amsthm}    
\usepackage{framed}
\usepackage{geometry} 
 \usepackage{dirtytalk}
\newtheorem{theorem}{Theorem}

\newtheorem{remark}{Remark}
\newtheorem{IHLQR}{IHLQR}
\newtheorem{problem}{Problem}
\newtheorem{proposition}{Proposition}
\newtheorem{lemma}{Lemma}
\usepackage{cuted}
\usepackage{tcolorbox}
\newtheorem{example}{Example}
\usepackage{amsmath}
\usepackage{xpatch}
 
\newtheorem*{proof*}{Proof}
\makeatletter
\def\ps@pprintTitle{%
  \let\@oddhead\@empty
  \let\@evenhead\@empty
  \let\@oddfoot\@empty
  \let\@evenfoot\@oddfoot
}
\makeatother

\begin{document}

\begin{frontmatter}



\title{Lyapunov-like Stability Inequality with an \textit{Asymmetric} Matrix and application to Suboptimal Linear Quadratic Control} %% Article title


\author{Avinash Kumar\\
avishimpu@gmail.com} %% Author name

%\affiliation{organization={},%Department and Organization
       %     addressline={}, 
          %  city={},
          %  postcode={}, 
          %  state={},
           % country={}}

%% Abstract
\begin{abstract}
The Lyapunov inequality is an indispensable tool for stability analysis in linear control theory. It provides a necessary and sufficient condition for the stability of an autonomous linear-time invariant system in terms of the existence of a symmetric positive-definite Lyapunov matrix. This work proposes a new variant of this inequality in which the constituent Lyapunov matrix is allowed to be asymmetric. After analysing the properties of the proposed inequality for a class of matrices, we derive new results for the stabilisation of linear systems. Subsequently, we utilize the developed results to obtain sufficient conditions for the suboptimal linear quadratic control design problem wherein addition to having an asymmetric Lyapunov matrix, which serves as a design matrix for this problem, we provide a characterization of the cost associated with the computed stabilizing suboptimal control laws. This characterization is done by deriving an expression for the upper bound on cost in terms of the initial conditions of the system. We demonstrate the applicability of the proposed results using two numerical examples- one for suboptimal control design for a linear time-invariant system and another for the consensus (state-agreement) protocol design for a multi-agent system, wherein we see how the asymmetry of the Lyapunov (design) matrix emerges as an inherent requirement for the problem.
\end{abstract}



%% Keywords

\end{frontmatter}



\section{Introduction}
\label{sec_intro}
Lyapunov theory is a vital tool for stability analysis in control theory, and its importance cannot be overstated. Lyapunov inequality, which results upon application of this theory to the linear time-invariant (LTI) systems with a quadratic positive-definite Lyapunov function, serves as the backbone for studying the stability properties of linear systems. It finds applications not only in the linear control theory but also in the robust and nonlinear control design \cite{khalil}, linear systems with disturbances \cite{khlebnikov2011optimization}, and model predictive control \cite{christofides2011networked}. It is the first linear matrix inequality (LMI) to appear in the control theory \cite{khlebnikov2020linear}; the field of LMIs has been advancing continually since then \cite{boyd1994linear}. The LMIs are a vital tool for the control design techniques today, thanks to easy solvability owing to the availability of computational resources \cite{khlebnikov2015linear}. As per the author's best knowledge, while utilizing Lyapunov inequality, the symmetry of the constituent Lyapunov matrix is always assumed. See the box labeled \say{\textit{The Symmetry Assumption}} for the reasons for this pervasive assumption. This work presents new results for a class of matrices wherein this assumption is dropped.
\begin{framed}
 The reasons for the presence of the symmetry assumption are vivid \cite{terrell2009stability}-
\begin{enumerate}
    \item the nice properties that symmetric matrices offer, namely, possession of real eigenvalues and the property of diagonalisability;
    \item the quadratic function, say $\mathbf{z}^TM \mathbf{z}$ with $\mathbf{z} \in \mathbb{R}^n$ and a symmetric $M \in \mathbb{R}^{n \times n}$, satisfies $\mathbf{z}^TM \mathbf{z}=\mathbf{z}^T M_{sym} \mathbf{z}$ where $M_{sym}=\frac{1}{2} \left( M+M^T \right)$ is the symmetric part of the matrix $M$; and
    \item the notion of the positive/negative-definiteness of the scalar quadratic from $\mathbf{z}^TM \mathbf{z}$ is easily transformed to an equivalent notion of positive/negative-definiteness of the constituent symmetric matrix $M$.
\end{enumerate}
 \begin{center}
\textit{The Symmetry Assumption}
\end{center} 
\vspace{-0.4 cm}
\end{framed}
  Linear-quadratic (LQ) optimal control design is a well-developed theory, wherein the optimal solution is obtained by solving the Riccati equation \cite{athans2007optimal, kumar2019some}. Nevertheless, the optimal solution may not always be implementable, because of the saturation constraints on the input, norm-boundedness requirements on the feedback-gain matrix, because this norm directly augments the maximum control effort \cite{datta2014feedback}, or the structural conditions on the feedback-gain matrix \cite{kumar2023suboptimal}. Thus, it becomes important to compute and analyse the suboptimal solutions for this problem. In the existing literature, to approach this suboptimal control design problem, an upper bound is pre-specified, and then sufficient conditions for the existence of suitable suboptimal control laws are derived \cite{jiao2019suboptimality}. These conditions impose that the initial conditions of the system reside in a certain set (ball); contrariwise, in this work, we develop results utilising the exact value of initial conditions and then argue that this latter technique can be extended straightforwardly to the former approach. Moreover, specifying the upper bound a priori may lead to infeasibilities if the upper bound is not chosen cautiously, especially for the networked control design problems, where the optimal cost is not known. In this work, we develop new results based on the proposed inequality and then utilize them to compute the stabilizing control law as well as the associated upper bound on the cost in terms of the initial conditions of the system \cite{kumar2023suboptimal}. Finally, using suitable numerical examples, we demonstrate the usage of these results to suboptimal LQ control design and to suboptimal consensus protocol design for a multi-agent system. We demonstrate in the latter example that the problem necessitates the computation of an asymmetric Lyapunov matrix, which serves as a design matrix for this problem. In summary, the contributions of this work are as follows.
 \begin{enumerate}
     \item It puts forward a new Lyapunov-like stability inequality for a class of matrices and the properties thereof;
%     \item it develops new conditions for the stabilisation of linear systems wherein the design matrix is allowed to be asymmetric; and
     \item it develops new results for the characterisation of the cost for suboptimal LQ control laws, where the upper-bound on the cost is computed in terms of the initial conditions, instead of specifying it a priori as generally done in the literature.
 \end{enumerate}
 The results proposed in this paper are expected to open up a new line of research in LMIs where the constituent Lyapunov matrices are not necessarily assumed to be symmetric thus allowing the enlargement of the search space for these design matrices. The remainder of the paper is organized as follows. In Section \ref{sec_2}, the background of the Lyapunov inequality, followed by the suboptimal linear quadratic control design problem, is presented; the remarks in this section collect the important observations regarding the results present in the literature, which which partly serve as a basis for the study carried out. Section \ref{sec_3} presents the main results of the paper -namely, a new Lyapunov-like inequality, a result for stabilisation of an LTI system, followed by the quantification of the suboptimal LQ control laws. Section \ref{sec_4} presents the numerical results, followed by the concluding remarks in Section \ref{sec_5}.

\textbf{\textit Notation:} Throughout this manuscript, scalars and scalar-valued functions are denoted by small ordinary alphabets, vectors are denoted by small bold alphabets, and capital ordinary alphabets denote matrices. $M \succ \mathbf{0} (\succeq \mathbf{0}) $ and $M \prec \mathbf{0}(\preceq \mathbf{0}) $  means that the matrix $M \in \mathbb{R}^{n \times n}$ is symmetric positive- (semi-)definite and symmetric negative-(semi-)definite respectively. `{$^\star$}' in the superscript of a quantity denotes the conjugation, `$T$' and `$H$' in the superscript denote transposition and the conjugate transposition respectively. $I_{n \times n}$ denotes an identity matrix of order $n$. A `$^*$' in superscript denotes the optimal counterpart of the quantity under consideration. $trace(M)$, $\lambda(M)$ and $\lambda_i(M)$ denote the trace, an eigenvalue and $i^{th}$ eigenvalue of the matrix $M$ respectively. $Re(m)$ denotes the real part of the quantity $m$.
%% Labels are used to cross-reference an item using \ref command.

\section{Background and Problem Formulation}

\label{sec_2}
In this section, we present the background of Lyapunov inequality, existing approaches to the suboptimal LQ control design problem, and the insights into the results that we shall develop, followed by the problem formulation.
\subsection{Lyapunov Inequality}
Let us first revisit the celebrated standard Lyapunov's stability theory as applicable to LTI systems. Consider the LTI autonomous system 
\begin{equation}
\dot{\mathbf{x}}(t)= \underline{A} \mathbf{x}(t)
\label{eq_AS}
\end{equation}
where $\mathbf{x}(t) \in \mathbb{R}^{n }$ and $  \underline{A}\in \mathbb{R}^{n \times n}$. The stability of the system is completely characterized by the signs of the real parts of the eigenvalues of the matrix $ \underline{A}$- if all the eigenvalues have strictly negative real parts, then the matrix $\underline{A}$ is called Hurwitz, and the system in \eqref{eq_AS} is globally asymptotically stable. The following theorem, which is well known in the literature, provides a necessary and sufficient condition for the matrix $ \underline{A}$ to be Hurwitz.
\begin{theorem} \cite{terrell2009stability}
    \label{lyap_thm}
    A matrix $ \underline{A} \in \mathbb{R}^{n \times n}$ is Hurwitz if and only if there exists a symmetric positive-definite matrix $P \succ \mathbf{0} $ such that 
    \begin{equation}
     \underline{A}^TP+P \underline{A} \prec \mathbf{0}.
    \label{eq_lya_o}
    \end{equation}
\end{theorem}
\begin{proof}
\normalfont
    The proof follows by choosing the global radially unbounded Lyapunov function $V(\mathbf{x}(t))=\mathbf{x}^T(t) P \mathbf{x}(t)$ where $P \in \mathbb{R}^{n \times n}$ is a symmetric positive-definite matrix and subsequently ensuring the strict negativity of the quantity $\frac{\partial V (\mathbf{x}(t))}{\partial \mathbf{x}(t)} \underline{A} \mathbf{x}(t)$  along the trajectories of the system.
\end{proof}
The inequality \eqref{eq_lya_o} is known as Lyapunov inequality in the literature and is a celebrated result in linear control theory.

\begin{remark}
\label{rem1}
  The Lyapunov matrix $P$  in Theorem \ref{lyap_thm} is intrinsically presumed to be symmetric, and this presumption persists in all the results that emanate from or are based around the Lyapunov inequality \eqref{eq_lya_o}.
\end{remark}
It is the symmetry assumption mentioned in Remark \ref{rem1}, which is prevalent due to reasons mentioned in Introduction (\textit{\say{The Symmetry Assumption}}), that this work attempts to mitigate. We shall see that the matrix $P$ is indeed a design matrix for the suboptimal linear quadratic problems- we shall have more to say about this when we introduce the new inequality and develop new results based upon this new inequality for suboptimal LQ control design.

\subsection{Suboptimal LQ Control Design}
Let us first revisit the standard linear-quadratic regulation (LQR) problem as applicable to the linear-time invariant systems.
\begin{IHLQR} \label{sec2_ihlqr}
{(Infinite-horizon LQR)}\\ 
{Compute the optimal feedback control law $\mathbf{u}^*(t)$ which minimizes the quadratic performance index:
\begin{equation}
\label{sec2_ihlqr_cost}
J=   \int_{t_0}^{ \infty} \left( \mathbf{x}^T(t)Q{\mathbf{x}(t)}+\mathbf{u}^T(t)R\mathbf{u}(t) \right) dt 
\end{equation}
subject to the system dynamics
\begin{align}
     &\dot{\mathbf{x}}(t)=A\mathbf{x}(t)+B\mathbf{u}(t) ; \label{eq_sys}\\
     \text{with initial conditions }& \mathbf{x}(t_0)=\mathbf{x}_0, \label{eq_ic}
\end{align}
and drives the states of the system to zero.
Here, $\mathbf{x}(t) \in \mathbb{R}^n$, $\mathbf{u}(t) \in \mathbb{R}^m$, the matrices $A$, $B$, $Q$ and $R$ are constant (time-invariant) with $Q \succeq \mathbf{0} \text{ and } R  \succ \mathbf{0}$.
}
\end{IHLQR}
The optimal solution for the IHLQR problem is given by 
$ 
\mathbf{u}^*(t) = -K \mathbf{x}(t)=-R^{-1} B^T P \mathbf{x}(t)
$ 
where $P \succ \mathbf{0}$ is the solution of the algebraic Riccati equation
$
A^TP+PA+Q - PBR^{-1}B^TP= \mathbf{0}.
$
As explained in the introduction, since the optimal control law cannot always be implemented, due to saturation/norm-boundedness conditions and/or due to structural conditions on the feedback-gain matrix $K$, the suboptimal version of the problem needs to be studied wherein the cost is ensured to be less than a positive bound. The following version of the suboptimal LQ control design is generally addressed in the literature.
    
\begin{problem}[$\gamma$-Suboptimal LQR]

Compute a control law of the form 
\begin{equation}
\mathbf{u}(t)= -R^{-1} B^{T} P \mathbf{x}(t)
\label{eq_u}
\end{equation}
such the system \eqref{eq_sys} with initial conditions \eqref{eq_ic} is regulated and the cost \eqref{sec2_ihlqr_cost} satisfies $J < \gamma.$
\label{sec2_sihlqr}
\end{problem}
The standard result on suboptimal LQR is based upon the Lyapunov inequality  \eqref{eq_lya_o}, as stated in Theorem \ref{thm_trent_slqr}, which we present without proof; see \cite{jiao2019suboptimality} (Theorem $5$) for more details.

\begin{theorem}
\label{thm_trent_slqr}
Let $\gamma>0$. A control law of the form \eqref{eq_u} is $\gamma$-suboptimal for the system \eqref{eq_sys} with initial conditions \eqref{eq_ic} and associated cost \eqref{sec2_ihlqr_cost}, if the matrix $P \succ \mathbf{0}$ satisfies
\begin{align}
    A^TP+PA-PBR^{-1}B^TP+Q \prec \mathbf{0}; \label{exist_slqr_ineq}
    \end{align}
    and the initial conditions \eqref{eq_ic} satisfy
    \begin{equation}
       \mathbf{x}_0^T P \mathbf{x}_0 < \gamma.  
       \label{eq_ic_set}
    \end{equation}
\end{theorem}

It is evident from Theorem \ref{thm_trent_slqr} that the matrix $P$ is indeed a design matrix for the suboptimal control laws of the form \eqref{eq_u} and is computed by solving the LMI \eqref{exist_slqr_ineq}. Moreover, the initial conditions are restricted to reside in a certain set \eqref{eq_ic_set} as dictated by the upper bound $\gamma$ which has been specified beforehand. Subsequently, Remark \ref{rem2} and Remark \ref{rem3} are worth noting regarding Theorem \ref{thm_trent_slqr}.
\begin{remark}[Symmetry of \text{$P$}]
 \label{rem2}
   The design matrix $P$ in the control law \eqref{eq_u} is assumed to be symmetric (this is inherited from Theorem \ref{lyap_thm} (see Remark \ref{rem1})) and thus limits the search space for $P$ to the set of symmetric positive-definite matrices.
\end{remark}
\begin{remark}[$\gamma$-Specification] 
\label{rem3}
The upper bound $\gamma>0$ has been pre-specified which restricts the initial conditions to reside in a certain set, and hence an incautious selection of $\gamma$ may lead to infeasibilities, a scenario quite possible in structured control design problems where the optimal cost is not known. (See \cite[Example $2$]{kumar2023suboptimal}).
\end{remark}
In this work, we propose results wherein the symmetry assumption on the design matrix $P$, as per Remark \ref{rem1}, is not imposed. This allows to search for the matrix $P$ over a wider set rather than restricting to the the set of symmetric positive-definite matrices. We develop these results for the case when the matrix $ \underline{A}$ in \eqref{eq_lya_o} is symmetric and hence propose a new LMI, which we call the Lyapunov-like Stability Inequality (\eqref{eq1_llsi}), where the symmetry assumption on the design matrix $P$ is dropped.
\begin{equation}
    \underline{A}P +P^T \underline{A} \prec \mathbf{0}.   \tag{LSI}
    \label{eq1_llsi}
\end{equation}

 It is easily verified that for a symmetric $P$ the inequality reduces to the Lyapunov inequality \eqref{eq_lya_o} with a symmetric $\underline{A}$. Furthermore, \eqref{eq1_llsi} is different from the second equivalent version of the Lyapunov inequality- $\underline{A}P+P\underline{A}^T \prec \mathbf{0}$- which has also been explored in the literature \cite{chilali1996h}. We use LSI, wherein $P$ is not necessarily symmetric, to establish the results for suboptimal control design, this scheme is in-line with the Remark \ref{rem2}. Furthermore, as mentioned in Remark \ref{rem3}, for the application of the existing results in the literature for this problem, one needs to specify the quantity $\gamma$ at the outset, whereas in the proposed results we characterize the suboptimality of the stabilizing control laws of the form \eqref{eq_u} by computing the upper bound on the cost in terms of the initial conditions of the system. Correspondingly, the problem under consideration in this work is Problem \ref{prob_2} as stated next.

 \begin{problem}{($\bar{\gamma} (\mathbf{x}_0)-$Suboptimal LQ Control)}
\label{prob_2}
Compute 
\begin{itemize}
    \item a control law of the form \eqref{eq_u}, where the matrix $P$ is not necessarily symmetric, and 
    \item the parameter $\bar{\gamma} (\mathbf{x_0} )>0$,
\end{itemize}  
such the system \eqref{eq_sys} is regulated and the associated cost \eqref{sec2_ihlqr_cost} satisfies $J < \bar{\gamma} (\mathbf{x_0} ).$
\end{problem}
We argue in the next section that a solution to Problem \ref{prob_2} can be easily cast to provide a solution to Problem \ref{sec2_sihlqr} if the bounds on the initial conditions (like \eqref{eq_ic_set}) are known.
 
\section{Main Results}
 \label{sec_3}
 
\subsection{Lyapunov-like Stability Inequality (LSI)}
 The next result is one of the important results of this work where we study one of the properties of the matrices satisfying \eqref{eq1_llsi}.

\begin{proposition}
\label{prop_asymm_P}
    Assume that ${\underline{A}} \in \mathbb{R}^{n \times n}$ is a symmetric negative-definite (and hence Hurwitz) matrix. If there exists a matrix ${P} \in \mathbb{R}^{n \times n}$ satisfying \eqref{eq1_llsi}, then $ Re( \lambda_i({P})) > 0$ $\forall$ $i \in \{1,2,3,\hdots, n\}$ and hence $trace(P)>0$.
\end{proposition}
 
     \begin{proof}
     \normalfont
     Let $\lambda$ be an arbitrary (possibly complex) eigenvalue of the matrix $P$ with an associated (possibly complex) eigenvector $\mathbf{v}$. Then
     \begin{align*}
         \underline{A}P+P^T \underline{A} & \prec \mathbf{0}\\
         \implies \mathbf{v}^H \left( \underline{A}P+P^T \underline{A} \right) \mathbf{v} &< 0\\
         \implies \mathbf{v}^H  \underline{A} P \mathbf{v} + \left( P \mathbf{v} \right)^H \underline{A} \mathbf{v} &<0 \\
         \implies \mathbf{v}^H \underline{A} \lambda \mathbf{v}+ \lambda^{\star}  \mathbf{v}^H \underline{A} \mathbf{v} & < 0\\
         \implies \left( \lambda+\lambda^{\star} \right) \mathbf{v}^H \underline{A} \mathbf{v} &< 0 \implies Re(\lambda)>0.
     \end{align*}
  Since, $\lambda$ is an arbitrary eigenvalue of $P$, we have that $Re(\lambda_i(P))>0$ $\forall i \in \{1,2,3, \hdots n\}$.
 \end{proof}
 
The next lemma shall be used for the quantification of the suboptimal control laws that we obtain in a later subsection.

 \begin{lemma}
    \label{lemma_t_ext}
Assume that ${\underline{A}} \in \mathbb{R}^{n \times n}$ is a symmetric negative-definite matrix. Suppose that there exist matrices $\tilde{P}\in \mathbb{R}^{n \times n}$ and $\hat{P}\in \mathbb{R}^{n \times n}$ satisfying 
\begin{align}
&{\underline{A}} \tilde{P}+ \tilde{P}^T {\underline{A}} -W \prec \mathbf{0} \text{ and},\label{ineqn_P_bar}\\
&{\underline{A}}\hat{P}+\hat{P}^T{\underline{A}}+W = \mathbf{0}\label{eqn_P};
\end{align}
where $W \succ \mathbf{0}$. Then $Re (\lambda_i(\tilde{P}+\hat{P})) >  {0}$ $\forall i \in \{1,2,3,  \hdots n\}$ and hence $trace \left(\tilde{P}+\hat{P} \right)>0$.
\end{lemma}

\begin{proof}
\normalfont
Adding \eqref{ineqn_P_bar} and \eqref{eqn_P}, we get
$$
 \underline{A} \left(\hat{P}+\tilde{P} \right) +\left(\hat{P}+\tilde{P} \right)^T \underline{A} \prec \mathbf{0}.
$$
Since $\underline{A}$ is symmetric negative-definite, the result follows from Proposition \ref{prop_asymm_P}.
\end{proof}
 
\subsection{Suboptimal LQ Control Design}
In this subsection, first we derive a new result for the stabilisation of linear system \eqref{eq_sys} with control input \eqref{eq_u}, wherein the design matrix $P\in \mathbb{R}^{n \times n}$ is expressed in terms of newly introduced matrix variables $X \succ \mathbf{0}$ and $Y$, an idea borrowed from \cite{polyak2021linear} (Theorem $3$).
\begin{proposition}[\textbf{\textit{Stabilisation of LTI Systems}}]
\label{prop_stabilising}
 If there exist a matrix $ X \succ \mathbf{0}$, $X \in \mathbb{R}^{n \times n}$ and a matrix $Y \in \mathbb{R}^{n \times n}$ satisfying \eqref{eq_new_lti} then the matrix $A-BR^{-1}B^TP$ is Hurwitz wherein $P= YX^{-1}$.
  \begin{align}
      XA^T+AX-Y^TBR^{-1}B^T-BR^{-1}B^TY \prec \mathbf{0}.
       \label{eq_new_lti}
\end{align}
\end{proposition}
\begin{proof}
\normalfont
Substituting $Y=PX$ in \eqref{eq_new_lti}, and followed by pre- and post-multiplication with $X^{-1}$ and subsequent simplification gives
 \begin{align*}
  \left(A-BR^{-1}B^TP \right)^TX^{-1}+ X^{-1} \left( A-BR^{-1}B^TP\right)\prec \mathbf{0}.
     \end{align*} 
     Since $X$ and hence $X^{-1}$ is positive-definite, using Theorem \ref{lyap_thm}, it follows that $A-BR^{-1}B^TP$ is Hurwitz.
\end{proof}
 Another result for a LTI system with an associated quadratic cost is stated below and shall be used in proving Theorem \ref{thm1} which is the main result of this subsection. 

\begin{lemma}
\label{lem_cost_lti}
Consider the system \ref{eq_sys} with cost \eqref{sec2_ihlqr_cost}. Assume that the matrix $A-BR^{-1}B^TP$ is Hurwitz. Then the cost $J$ associated with the control input \eqref{eq_u} is finite and is given as $J= \mathbf{x}_0^T Z \mathbf{x}_0$ where $Z \succ \mathbf{0}$ is the unique solution of the equation \eqref{eq_cost_z}.
\begin{align}
    & A^TZ+ZA+Q-P^TBR^{-1}B^TZ  -ZBR^{-1}B^TP+P^TBR^{-1}B^TP=\mathbf{0}.
    \label{eq_cost_z}
\end{align}
\end{lemma}

\begin{proof}
\normalfont
    Substituting \eqref{eq_u} in \eqref{sec2_ihlqr_cost}, we obtain
    $$
    J= \int_{t_0}^{\infty} \left(\mathbf{x}^T(t)  \left(Q+ P^T B R^{-1}B^T P \right)\mathbf{x}(t) \right) dt.
    $$
    Since the closed-loop matrix $A-BR^{-1}B^TP$ is Hurwitz, it follows (\cite{jiao2019suboptimality}(Lemma $3$)) that the cost is $J=\mathbf{x}_0^TZ \mathbf{x}_0$ where $Z \succ \mathbf{0}$ is the unique solution of 
    \begin{align*}
    &\left(A-BR^{-1}B^TP \right)^T Z +Z \left(A-BR^{-1}B^T P \right)  + P^T BR^{-1} B^T P = \mathbf{0}\\
   \implies  & A^TZ+ZA+Q-P^TBR^{-1}B^TZ \nonumber -ZBR^{-1}B^TP+P^TBR^{-1}B^TP=\mathbf{0}. 
    \end{align*} 
\end{proof}
The next result is the most important contribution of this work where we propose sufficient conditions for obtaining solution to Problem \ref{prob_2} by deriving an expression for the upper bound on the cost in terms of the matrix variables and the initial conditions of the system.

\begin{theorem}{(\textbf{Quantification of Suboptimality})}
\label{thm1}
Assume that there exist matrices $X\succ \mathbf{0} \in\mathbb{R}^{n \times n}$, $W \succ \mathbf{0} \in \mathbb{R}^{n \times n}$; and matrices $Y \in \mathbb{R}^{n \times n}$, $\underline{P} \in \mathbb{R}^{n \times n}$, such that:
    \begin{align}
     XA^T+AX-Y^TBR^{-1}B^T-BR^{-1}B^TY \prec \mathbf{0};
     \label{eqn_lmi_stability}
\end{align}  \text{and}
    \begin{align}
\label{eqn_lmi_upp_bd}
    \begin{bmatrix}
   {A}^T \underline{P}+\underline{P}^TA - {Q} +W &  \left( I+ Y^TA\right)^T & \left(\underline{P}^T {B} \right)\\
    \left( I+ Y^TA\right) & X & \mathbf{0}_{n \times m}\\
  \left(  \underline{P}^T {B} \right)^T & \mathbf{0}_{m\times n} & {R}
    \end{bmatrix}\succ \mathbf{0} ;
    \end{align}
Then the matrix ${A}-{B}{R}^{-1}{B}^T{P}$, with $P=YX^{-1}$, is Hurwitz. If in addition  ${A}-{B}{R}^{-1}{B}^T{P}$ is symmetric and $\exists$ $\hat{P}$ such that 
\begin{equation}
\label{eq_Ph}
    \left( {A}-{B}R^{-1}B^T {P} \right) \hat{P}+ \hat{P}^T \left( {A}-{B}R^{-1}B^T  {P} \right)+W=\mathbf{0},
\end{equation}
then the control law \eqref{eq_u} solves Problem \ref{prob_2} with $\bar{\gamma}(\mathbf{x
}_0)= trace\left( {P}  -\underline{P}   + \hat{P} \right) \mathbf{x}_0^T\mathbf{x}_0$.
\end{theorem}
\begin{proof}
 The Hurwitz property of ${A}-{B}{R}^{-1}{B}^T{P}$ follows from \eqref{eqn_lmi_stability} using Proposition \ref{prop_stabilising}.
Using the Schur complement lemma, LMI \eqref{eqn_lmi_upp_bd}, can be written as below.
    \begin{align*} 
     A^T \underline{P}+ \underline{P}^T A -Q+W- X^{-1}-A^TY X^{-1}-X^{-1}Y^TA-A^T YX^{-1}Y^TA- \underline{P}^TBR^{-1}B^T \underline{P} & \succ \mathbf{0}.
    \end{align*}
    Since $X$ and hence $X^{-1}$ is positive-definite the quantity $A^TYX^{-1}Y^TA$ is positive-definite. With this information and using $P=YX^{-1}$ and $P^T=X^{-1}Y^T$, we obtain
    \begin{align*}
         A^T \left(\underline{P}-P \right)+ \left(\underline{P}-P \right)^T A -Q+W- \underline{P}^T BR^{-1}B^T \underline{P} \succ \mathbf{0}.
    \end{align*}
    Adding the positive-definite term $\left( P- \underline{P}\right)^T B R^{-1}B^T \left(P- \underline{P} \right)$ to the left hand side does not change the sign of the inequality and we get 
\begin{align*}
    A^T \underline{P}+ \underline{P}^TA-Q+W -P^T A -A^T P +P^T BR^{-1}B^TP -P^TB R^{-1}B^T \underline{P}-\underline{P}^TBR^{-1}B^TP \succ \mathbf{0}.
\end{align*}
Substituting for the quantity $-Q$ using \eqref{eq_cost_z}, we have
\begin{align*}
    &A^T \underline{P}+\underline{P}^TA+W-P^TA-A^TP+P^TBR^{-1}B^TP -P^TBR^{-1}B^T \underline{P}\\
    & \hspace{0.4 cm}-\underline{P}^TBR^{-1}B^TP+A^TZ+ZA-P^TBR^{-1}B^TZ-ZBR^{-1}B^TP+P^TBR^{-1}B^TP \succ \mathbf{0}\\
    \implies & A^T \left( \underline{P}-P+Z \right)+ \left(\underline{P}^T-P^T+Z \right)A+W+P^TBR^{-1}B^TP-P^TBR^{-1}B^T \underline{P}\\
    & \hspace{0.4 cm}-\underline{P}^T BR^{-1}B^T P -ZBR^{-1}B^TP-P^TBR^{-1}B^TZ+P^TBR^{-1}B^TP \succ \mathbf{0}\\
    \implies & \left(A -BR^{-1}B^TP\right)^T \left( -P+\underline{P}+Z \right)+\left( -P+\underline{P}+Z \right)^T \left(A-BR^{-1}B^TP \right)+W \succ \mathbf{0}\\
    \implies & \left(A- BR^{-1}B^TP \right)^T\left(P- \underline{P}-Z \right)+ \left( P-\underline{P}-Z \right)^T \left(A-BR^{-1}B^TP \right)-W \prec \mathbf{0}.
\end{align*}
    Since $A-BR^{-1}B^TP$ is symmetric, it follows from Lemma \ref{lemma_t_ext} that
    \begin{align*}
   trace \left(P-\underline{P}-Z +\hat{P} \right) >0  & \implies  trace(Z) < tr(P- \underline{P}+ \hat{P})  \\
    & \implies  \lambda_{max}(Z) < tr(P-\underline{P}+\hat{P} )  \\
     &\implies  \mathbf{x}_0^T Z \mathbf{x}_0 < \lambda_{max}(Z)\mathbf{x}_0^T \mathbf{x}_0  < trace \left(P-\underline{P}+ \hat{P} \right) \mathbf{x}_0^T \mathbf{x}_0.
    \end{align*}
    where the last inequality follows from the fact that the matrix $Z$ is symmetric positive-definite.
    Finally since $J=\mathbf{x}_0^T Z \mathbf{x}_0$ (Lemma \ref{lem_cost_lti}), it follows that $J < trace \left(P-\underline{P}+ \hat{P} \right) \mathbf{x}_0^T \mathbf{x}_0$.
  \end{proof}

\begin{remark}[Applicability of Theorem \ref{thm1}]
\label{rem_4}
   While deploying Theorem \ref{thm1}, the symmetry condition (on the matrix $\underline{A}$ in \eqref{eq1_llsi}) reflects upon the \textit{closed-loop system matrix} $A-BR^{-1}B^TP$ and not upon the open-loop system matrix $A$. Thus, the applicability of the results is not limited to the systems with a symmetric open-loop system matrix $A$.
\end{remark}  
We demonstrate the fact stated in Remark \ref{rem_4} in Example \ref{ex1} where the open-loop system matrix $A$ is not symmetric. As a final observation, as stated in the introduction, generally sufficient conditions for suboptimality are derived for the initial conditions residing in a certain set(see Theorem \ref{thm_trent_slqr}, equation \eqref{eq_ic_set}), instead of carrying out the computation for any arbitrary initial condition as is done in this work. It is easily verified that it is not possible to extend the results of the former approach to the latter but the extension the other way round is straightforward as seen in Remark \ref{rem_last}.

\begin{remark}[Specification of Initial Conditions]
Assume that the initial conditions of the system \eqref{eq_sys} belong to the $\alpha-$ball, that is $\{\mathbf{x}_0: \mathbf{x}_0^T \mathbf{x}_0 <\alpha >0 \}$. Then, Theorem \ref{thm1} gives an upper bound on the cost as $J<  \alpha \ trace\left( {P}  -\underline{P}   + \hat{P} \right).$
\label{rem_last}
\end{remark}
A summarised comparison of the various aspects of the proposed technique with the existing state-of-the-art (SOTA) techniques is presented in Table \ref{tab_1} from the viewpoint of stability analysis and in Table \ref{tab_2} from the viewpoint of suboptimal LQR design.
 
\begin{table}[]
\centering
\begin{tabular}{|c |c| c|} 
 \hline
 Characteristic & Existing (Lyapunov inequality) & Proposed (\eqref{eq1_llsi})  \\ [0.5ex] 
 \hline
 LMI &  $\underline{A}^TP+P\underline{A} \prec \mathbf{0}$ &  $\underline{A}P+P^T\underline{A}\prec \mathbf{0}$   \\ 
 \hline
 Conditions on $P$& $P \succ \mathbf{0}$ & $Re(\lambda_i(P))>0$ $\forall i$  \\
 \hline
 Conditions on $\underline{A}$ & $\underline{A}$ is Hurwitz & $\underline{A} \prec \mathbf{0}$\\
 \hline
 Nature & Necessary and sufficient & Sufficient \\
 \hline
\end{tabular}
\caption{Comparison with SOTA (stability perspective) }
\label{tab_1}
\end{table}

\begin{table*}

\resizebox{\textwidth}{!}{
\begin{tabular}{|c |c| c|} 
 \hline
 Characteristic & Existing (\cite{jiao2019suboptimality}) & Proposed  \\ [0.5ex] 
 \hline
 LMI & 
     $A^TP+PA-PBR^{-1}B^TP+Q \prec \mathbf{0}$ & $ \begin{array} {lcl} &\left(A- BR^{-1}B^TP \right)^T  \left(P- \underline{P}-Z \right)+& \\ & \hspace{1.4 cm} \left( P-\underline{P}-Z \right)^T \left(A -BR^{-1}B^TP \right)-W\prec \mathbf{0} & \end{array}$\\
 \hline
 Nature & Sufficient & Sufficient \\
 \hline
 Conditions on $P$ & $P \succ \mathbf{0}$ & $Re(\lambda_i(P))>0$ $\forall$ $i$ \\
 \hline
 Applicability & \textbf{ $P$ must be symmetric positive-definite} & \bf{$P$ need not be symmetric positive-definite} \\
 \hline
Conditions on closed-loop matrix  & Hurwitz &  $A-BR^{-1}B^TP \prec \mathbf{0}$\\
\hline
Bound on cost & pre-specified $\gamma$ with $\mathbf{x}_0^T P \mathbf{x}_0 < \gamma >0 $& $\bar{\gamma}{(\mathbf{x}_0)} $ \\
\hline
Suboptimality quantification & \bf{absent}& \bf{$\bar{\gamma} (\mathbf{x}_0) $ computed explicitly} \\
 \hline
Extensibility & \textbf{not possible} & \textbf{straightforward} (see Remark \ref{rem_last})\\
\hline
\end{tabular}
}
\caption{Comparison with SOTA (suboptimal LQR perspective) }
\label{tab_2}
\end{table*}

\section{Numerical Examples}
\label{sec_4}
(All LMI-based problems are numerically solved using the CVX toolbox \cite{grant2008matlab}.)

In this section, we demonstrate the usability of the proposed results for two examples- suboptimal control design for an unstable LTI system in Example \ref{ex1} and suboptimal consensus protocol design for a multi-agent system in Example \ref{ex2}. Moreover, to ensure that the closed-loop matrix $A-BR^{-1}B^TP$ has real eigenvalues, we use the LMI \eqref{eq_real}, which has been derived by characterisation of the stability regions \cite{chilali1996h,kumar2019elimination}, in conjunction with the proposed LMIs.

 \begin{align}
\small{
 \label{eq_real}
 \begin{bmatrix}
     \mathbf{0} &  \left( AX-XA^T-BR^{-1}B^TY+Y^TBR^{-1}B^T \right) \\
             \left( AX-XA^T-BR^{-1}B^TY+Y^TBR^{-1}B^T \right)^T  & \mathbf{0}
           \end{bmatrix} \prec \mathbf{0}.
}%
\end{align}

\begin{example}{(Unstable LTI system)}
\label{ex1}
\normalfont
Consider the system \eqref{eq_sys} with $$A= \begin{bmatrix} 1 & 2 \\ 0 & 2\end{bmatrix}, B= \begin{bmatrix}4 & 2 \\ 0 & 2 \end{bmatrix} \text{ and } \mathbf{x}_0= \begin{bmatrix} 0.1  &-0.2
\end{bmatrix}^T;$$ and the associated cost \eqref{sec2_ihlqr_cost} with $Q=10I_{2 \times 2}$ and $R= 0.05I_{2 \times 2}$. The optimal control law for this problem is $$\mathbf{u}^*(t)=- \begin{bmatrix}
  13.6444 &   -4.5552 \\
    4.5446 &  14.4257
\end{bmatrix} \mathbf{x}^*(t)$$ with the optimal cost $J^*=0.0207$.
For this example, solving the LMIs \eqref{eqn_lmi_stability}, \eqref{eqn_lmi_upp_bd} and \eqref{eq_real} gives
$$ 
    X= \begin{bmatrix}
   101.0514   &-0.1038 \\
   -0.1038 & 104.9348
    \end{bmatrix}, Y=\begin{bmatrix}
    1.7197  & -1.3530 \\
   -1.7190  &  9.3904
    \end{bmatrix},$$
    $$
    W= \begin{bmatrix}
      109.3360  &  1.8990 \\
    1.8990  & 93.1398
    \end{bmatrix}, \underline{P}= \begin{bmatrix}
     0  & -0.0136 \\
   -0.1104  &  0.9284
    \end{bmatrix},
$$
yielding $$P= \begin{bmatrix}         0.0170  & -0.0129 \\
   -0.0169  &  0.0895
    \end{bmatrix}$$ and the suboptimal control law  $$\mathbf{u}(t)= - \begin{bmatrix}         1.3604 &  -1.0302 \\
    0.0034 &   3.0638 \end{bmatrix} \mathbf{x}(t).$$ The matrix $\hat{P}$ in \eqref{eq_Ph} is computed to be $$\hat{P}=\begin{bmatrix} 12.2894 & 0.2092 \\
   0.1955 & 11.2823 
\end{bmatrix}$$ yielding $\bar{\gamma} (\mathbf{x}_0)=1.1374$. The actual cost is computed to be $J=0.0627$.
\end{example}

Next, we demonstrate the application of the proposed technique to a suboptimal consensus protocol design problem, where we see the significance of having an asymmetric design matrix.   
\begin{example}{(Suboptimal Consensus Protocol)}
\label{ex2}
\normalfont
\begin{figure}[h]
    \centering
    \includegraphics[scale=0.8]{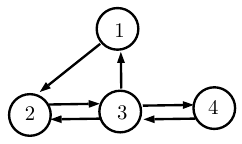}
    \caption{Interaction topology for Example \ref{ex2}}
    \label{fig_ex2}
\end{figure}
Consider a group of four scalar unstable agents with individual dynamics $\dot{x}_i(t)=x_i(t)+u_i(t)$, $i\in \{1,2,3,4\}$ with $x_1(0)=0.9$, $x_2(0)=1.0$, $x_3(0)=0.6$ and $x_4(0)=0.4$, communicating over the topology as shown in Figure \ref{fig_ex2} where the relative information $x_i(t)-x_j(t)$ is being shared among the agents and arrows indicate the direction of this information exchange. The objective is to compute a control protocol $\mathbf{u}(t)= \begin{bmatrix}u_1(t)& u_2(t)&u_3(t)&u_4(t) \end{bmatrix}^T$ to achieve a consensus (state-agreement), that is, $|x_i(t)-x_j(t)| \rightarrow 0$ as $t \rightarrow  \infty$ $\forall i \neq j, i,j \in \{1,2,3,4 \}$. Let the associated cost functional be considered as $ 
J= \int_0^{\infty}  10 \big(\left(x_1(t)-x_2(t) \right)^2+ \left( x_2(t)-x_3(t) \right)^2+ \left(x_3(t)-x_4(t) \right)^2 \big)+ 0.01 \left( u_1^2(t)+u_2^2(t)+u_3^2(t)+u_4^2(t) \right)    dt.
$ 

We formulate the problem as an LQ control design problem with structural conditions on the feedback-gain matrix \cite{kumar2023suboptimal}. Let us define error-vector $$\mathbf{e}(t)= \begin{bmatrix}
    e_1(t) \\ e_2(t) \\ e_3(t)
\end{bmatrix}= \begin{bmatrix}
    x_1(t)-x_2(t) \\ x_2(t)-x_3(t)\\ x_3(t)-x_4(t)
\end{bmatrix}.$$ 
The consensus achievement (that is state-agreement) is equivalent to the regulation of errors: $|e_i(t)-e_j(t)| \rightarrow 0$ for $i \neq j, i,j \in \{1,2,3\}$ as $t \rightarrow \infty$. With the information exchange as per the information topology Figure \ref{fig_ex2}, the feedback-gain matrix $K$ must have the structure as in \eqref{st_K}.

\begin{equation}
    \mathbf{u}(t)= -K \mathbf{e}(t)= \begin{bmatrix}
     k_1 & k_1 &0 \\
     k_2 & k_3 & 0 \\
     0 & k_4 & k_5 \\
     0& 0& k_6 
    \end{bmatrix} \mathbf{e}(t).
    \label{st_K}
    \end{equation}
    where $k_i$s are non-zero real parameters. Then, the problem is equivalent to computing a stabilising control law $\mathbf{u}=-K \mathbf{e}$ with a structured $K$ as in \eqref{st_K} for the system $\dot{\mathbf{e}} = A \mathbf{e} + B \mathbf{u}$, where $A=diag(1,1,1)$ and $$B= \begin{bmatrix}
        1 &-1&0&0\\
        0& 1 &-1& 0\\
        0 & 0& 1 & -1
    \end{bmatrix},$$ $\mathbf{e}_0 =\begin{bmatrix}-0.1 & 0.4 & 0.2
\end{bmatrix}^T$ with an associated cost $J= \int_{0}^{\infty} 10 \mathbf{e}^T \mathbf{e}+0.01 \mathbf{u}^T \mathbf{u} dt$ . With $K =- R^{-1}B^T P$, it is easily verified that $K$ has the structure in \eqref{st_K} only if the matrix $P$ has the structure in \eqref{st_P}.

    \begin{equation}
        P= \begin{bmatrix}
            p_1 & p_1 & 0\\
            0& p_2 & 0 \\
            0 & 0& p_3
        \end{bmatrix}
        \label{st_P}
    \end{equation}
    where $p_i$s are non-zero real numbers. Thus, for this problem, the design matrix $P$ has to be asymmetric, rendering the standard symmetric positive-definite Lyapunov matrix unusable and thus accentuating the relevance of the results presented in this work. To use Theorem \ref{thm1} for this problem, we translate the structural conditions on $P$ in \eqref{st_P} onto the matrices $X$ and $Y$ using expression $P=Y X^{-1}$- by imposing $X$ to be a scalar matrix (that is a diagonal matrix with all entries being identical) and $Y$ to have a structure identical to that of $P$ (\eqref{st_P}). With these conditions, solving LMIs \eqref{eqn_lmi_stability}, \eqref{eqn_lmi_upp_bd} and \eqref{eq_real} yields a solution detailed in \eqref{ex2_sol}. With this solution, $$K= \begin{bmatrix}
   -3.4981  & -3.4981   &      0 \\
    3.4981 &  -6.9962 &        0\\
         0  & 10.4943 & -10.4943\\
         0      &   0&   10.4943
         \end{bmatrix}$$ and
 $\bar{\gamma} (\mathbf{x}_0)=1.8426$. The cost associated with this protocol is computed to be $J= 0.1407$.
\end{example}

 \begin{equation}
\begin{gathered} 
 X= 19.4505 \  I_{3 \times 3}, Y=\begin{bmatrix}
    0.6804   & 0.6804 &        0 \\
   0   & 2.0412 &        0\\
   0    &     0  &  2.0412
  \end{bmatrix},
    P= \begin{bmatrix}
          0.0350 &   0.0350 &         0\\
   0   & 0.1049     &   0 \\
   0  &       0  &  0.1049   \end{bmatrix},\\
         \underline{P}= \begin{bmatrix}
             0.0501  &  0.0528  &  0.0207 \\
    0.0418   & 0.0919  &  0.0442 \\
    0.0195   & 0.0459 &   0.0678
         \end{bmatrix}, \text{ and } \hat{P}=\begin{bmatrix}
 4.5077 & 1.4237 & 0.74569\\
 1.4237& 2.4549&   1.2849\\
0.74569& 1.2849&  1.7771 
          \end{bmatrix}.
\end{gathered}
. \hrulefill
  \label{ex2_sol}
          \end{equation}

\section{Conclusion}
\label{sec_5}
This work proposes a new Lyapunov-like stability inequality for a class of matrices wherein the constituent Lyapunov matrix is not assumed to be symmetric at the outset. The usage of the proposed stability and stabilisation conditions is demonstrated in the numerical examples, wherein the control laws as well as the bounds on the associated cost are computed. It is clear from the second numerical example that an asymmetry of the design matrix may be an inherent requirement in the problem definition, thus necessitating the computation of an asymmetric Lyapunov matrix. The future work includes the exploration of this fact to more consensus protocol design problems and to other structural control design problems in networked control.


%%
%\bibliographystyle{elsarticle-harv} 
\bibliographystyle{ieeetr}
 \bibliography{ref2}



\end{document}